\documentclass[12pt]{article}
\usepackage{amssymb}
\usepackage[mathscr]{eucal}
\usepackage{amsthm,amsmath,anysize}
\def\a{\alpha}

\def\l{\lambda}
\def\g{\gamma}
\def\0{\bar{0}}
\def\1{\bar{1}}
\def\e{\epsilon}
\def\d{\delta}
\def\g{\mathfrak{g}}

\def\g{{\mathfrak g}}
\def\l{{\lambda}}

\def\A{\mathcal A}
\newtheorem{lemma}{Lemma}[section]
\newtheorem{theorem}[lemma]{Theorem}

\newtheorem{proposition}[lemma]{Proposition}

\newtheorem{corollary}[lemma]{Corollary}

\hoffset \voffset \oddsidemargin=60pt \evensidemargin=45pt
\title{ The simplicity of Kac modules  for the quantum superalgebra  $U_q(gl(m,n))$}
\author{
Chaowen Zhang\\ Department of
Mathematics,\\ China university
 of Mining and Technology,\\ Xuzhou, 221116, Jiang Su, P. R. China}
\date{ }

\begin{document}
\maketitle


\section{Introduction}  Let $\g=\g_{\0}\oplus \g_{\1}$ be the general linear Lie superalgebra over the complex number field $\mathbb C$. The quantum superalgebra $U_q(\g)$ in the present paper was defined by R. Zhang \cite{zh}.  The Kac module $K(M)$ is the $U_q(\g)$-module induced from  a simple $U_q(\g_{\0})$-module $M$. Assume $M$ is a weighted $U_q(\g_{\0})$-module which is generated by a primitive vector of weight $\l$. Then $\l$ is called
typical if $K(M)$ is simple. The typical weights in both generic case and the case where $q$ is a primitive root of unity were first studied  in \cite{zh}. Also in \cite{kw},  a sufficient condition for the typicality is given in generic case.\par One of the main goals of the present paper is to determine the typical weights.  We prove that in the case where $K(M)$ is  weighted, the typical weights are determined by a polynomial. Then we determine the polynomial using the method provided by \cite{z2}. Let us note that our polynomial coincides with one given in \cite{zh}, despite the fact that the order of the product for the  elements $F_{ij}$($(i,j)\in\mathcal I_1$) used in \cite{zh} to define the polynomial  is completely different from ours.\par
The paper is organized as follows. Sec. 3 is the preliminaries. In Sec. 4, we give some identities in $U_q(\g)$. In Sec. 5 we discuss the simplicity of the Kac modules, which is determined by a polynomial. The polynomial is determined in Sec. 6. In Sec. 7, we study the simple modules in the case where $q$ is a $l$th root of unity. We prove that, under certain conditions, the algebras $u_{\eta,\chi}(\g_{\0})$ and $u_{\eta,\chi}$ are Morita equivalent.
\section{Notations}
Throughout the paper we use the following notation.\par
$[1, m+n)\quad\quad             =\{1,2,\cdots, m+n-1\}.$\par
$[1,m+n]\quad\quad =\{1,2,\cdots, m+n\}.$\par
$A^{m+n}\quad\quad \quad$the set of all $m+n$-tuples $z=(z_1\dots z_{m+n})$ with $z_i\in A$ for all $i=1,\cdots, m+n$\par
$\mathcal I_0\quad\quad\quad\quad\quad=\{(i,j)|1\leq i<j\leq m\quad \text{or}\quad m+1\leq i<j\leq m+n\}$\par
$\mathcal I_1\quad\quad\quad\quad\quad=\{(i,j)|1\leq i\leq m<j\leq m+n\}$\par
$\mathcal I\quad\quad\quad\quad\quad\quad=\mathcal I_0\cup\mathcal I_1$\par
$A^B\quad\quad\quad\quad\quad$   the set of all  tuples
$\psi=(\psi_{ij})_{(i,j)\in B}$ with  $\psi_{ij}\in A$, where $B=\mathcal I_0$ or $B=\mathcal I_1$\par
$\mathcal A\quad\quad\quad\quad\quad=\mathbb C[q]$ where $q$ is an indeterminate\par
$h(V)\quad\quad\quad\quad$    the set of all  homogeneous elements  in a $\mathbb Z_2$-graded vector space $V=V_{\0}\oplus V_{\1}$\par
$\bar x\quad\quad\quad\quad\quad\quad$    the parity of the  homogeneous element $x\in V=V_{\0}\oplus V_{\1}$.\par
$U(L)\quad\quad\quad\quad\quad$the universal enveloping superalgebra for the Lie superalgebra $L$.\par

\section{The quantum deformation of $gl(m,n)$}

The general linear Lie superalgebra
 $\g=\g_{\0}\oplus \g_{\1}$ has the standard  basis(\cite{k3}) $e_{ij}$, $1\leq i,j\leq m+n$. We denote $e_{ji}$ with $i<j$ also by $f_{ij}$. Then we get $\g=\g_{-1}\oplus \g_{\0}\oplus\g_{1}$, where $$\g_{1}=\langle e_{ij}|(i,j)\in\mathcal I_1\rangle \quad \g_{-1}=\langle f_{ij}|(i,j)\in\mathcal I_1\rangle .$$  The parity of the basis elements is given
by $$\bar{e}_{ij}=\bar f_{ij}=\begin{cases} \bar 0,&\text{if $(i,j)\in \mathcal I_0$ or $i=j$}\\\1, &\text{if $(i,j)\in \mathcal I_1$.}\end{cases}$$
 Let $H=\langle e_{ii}|1\leq i\leq m+n\rangle.$ Then the set of positive roots of
   $\g$ relative to $H$ is $\Phi^+=\Phi^+_{0}\cup\Phi^+_{1},$ where $$\Phi^+_{0}=\{\e_i-\e_j|(i,j)\in \mathcal I_0\},  \Phi^+_{1}=\{\e_i-\e_j|(i,j)\in \mathcal I_1\}.$$
Let $\Lambda=\mathbb Z\e_1+\cdots +\mathbb Z\e_{m+n}\subseteq H^*$. There is a symmetric bilinear form defined on $\Lambda$ as follows(\cite{z1}):$$(\e_i,\e_j)=\begin{cases}\d_{ij},&\text{if $i<m$}\\-\d_{ij},&\text{if $i>m$.}\end{cases}$$

Let $q$ be an indeterminate  over $\mathbb C$. Then the quantum supergroup $U_q(\g)$(see \cite[p.1237]{zh}) is defined as the $\mathbb C(q)$-superalgebra with the generators $K_j,K^{-1}_j,E_{i,i+1},F_{i,i+1}$,  $i\in [1,m+n)$, and  relations
$$(R1)\quad K_iK_j=K_jK_i, K_iK_i^{-1}=1,$$
$$(R2)\quad K_iE_{j,j+ 1}K_i^{-1}=q_i^{(\delta_{ij}-\delta_{i,j+ 1})}E_{j,j+ 1}, \quad K_iF_{j,j+ 1}K_i^{-1}=q_i^{-(\delta_{ij}-\delta_{i,j+ 1})}F_{j,j+ 1},$$
$$(R3)\quad [E_{i,i+1},F_{j,j+1}]=\delta_{ij}\frac{K_iK^{-1}_{i+1}-K^{-1}_iK_{i+1}}{q_i-q^{-1}_i},$$
$$(R4)\quad E_{m,m+1}^2=F_{m,m+1}^2=0,$$$$(R5)\quad E_{i,i+1}E_{j,j+1}=E_{j,j+1}E_{i,i+1}, \quad F_{i,i+1}F_{j,j+1}=F_{j,j+1}F_{i,i+1}, |i-j|>1,$$
$$(R6)\quad E_{i,i+1}^2E_{j,j+1}-(q+q^{-1})E_{i,i+1}E_{j,j+1}E_{i, i+1}+E_{j,j+1}E_{i,i+1}^2=0\quad (|i-j|=1, i\neq m),$$
$$(R7)\quad F_{i,i+1}^2F_{j,j+1}-(q+q^{-1})F_{i,i+1}F_{j,j+1}F_{i, i+1}+F_{j,j+1}F_{i,i+1}^2=0\quad (|i-j|=1, i\neq m),$$
$$(R8)\quad [E_{m-1,m+2}, E_{m,m+1}]=[F_{m-1,m+2}, F_{m,m+1}]=0,$$ where $$q_i=\begin{cases}q,&\text{if $i\leq m$}\\q^{-1},&\text{if $i>m$.}\end{cases}  $$
Most often, we shall  use $E_{\a_i}$(resp. $F_{\a_i}$; $K_{\a_i}$) to denote $E_{i,i+1}$(resp. $F_{i,i+1}$; $K_iK_{i+1}^{-1}$) for $\a_i=\e_i-\e_{i+1}$.
\par
Remark: (1) For each pair of indices $(i,j)\in\mathcal I$, the notation  $E_{ij}, F_{ij}$ are defined by $$\begin{matrix} E_{ij}=E_{ic}E_{cj}-q_c^{-1}E_{cj}E_{ic}, \\ F_{ij}=-q_cF_{ic}F_{cj}+F_{cj}F_{ic},\end{matrix}\quad i<c<j.$$ The relation (R2) then implies that, for $s\in [1,m+n], (i,j)\in\mathcal I$, $$\begin{aligned} &K_sE_{ij}K_s^{-1}&=&q_s^{\d_{si}-\d_{sj}}E_{ij}\\&K_sF_{ij}K^{-1}_s&=&q_s^{-(\d_{si}-\d_{sj})}F_{ij}.\end{aligned}$$
 (2) The parity of the elements  $E_{ij}, F_{ij}, K_s^{\pm 1}$ is defined by  $\bar{E}_{ij}=\bar F_{ij}=\bar{e}_{ij}\in\mathbb Z_2$, $\bar K_s^{\pm 1}=\0$. \par (3) The bracket product  in  $U_q(\g)$ is defined by  $$[x,y]=xy-(-1)^{\bar x\bar y}yx, x,y\in h(U_q(\g)).$$
 A bijective (even)$\mathbb F$-linear map $f$ from an $\mathbb F$-superalgebra $\mathfrak A$ into itself is called an anti-automorphism(resp. $\mathbb Z_2$-graded anti-automorphism) if $$f(xy)=f(y)f(x)(\text{resp. $f(xy)=(-1)^{\bar x\bar y}f(y)f(x)$})$$ for any $x,y\in h(\mathfrak A)$.\par It is easy to show that
 \begin{lemma}\cite{z1,zh} There is an anti-automorphism $\Omega$ and a $\mathbb Z_2$-graded anti-automorphism $\Psi$ of $U_q(\g)$  such that $$\Omega (E_{\a_i})=F_{\a_i}, \Omega (F_{\a_i})=E_{\a_i}, \Omega (K_j)=K_j^{-1},\Omega (q)=q^{-1}$$ $$\Psi (E_{\a_i})=E_{\a_i}, \Psi (F_{\a_i})=F_{\a_i}, \Psi (K_j)=K_j,\Psi (q)=q^{-1},$$for all $i\in [1,m+n), j\in [1,m+n]$.\end{lemma}  From the lemma it is easily seen  that  $$\Omega (E_{ij})=F_{ij},  \Psi (E_{ij})=q^zE_{ij}, \Psi (F_{ij})=q^zF_{ij}, z\in\mathbb Z$$  for any $(i,j)\in\mathcal I$.\par

We abbreviate $U_q(\g)$ to $U_q$ in the following.
\section{Some formulas in $U_q$}
In this section we present some formulas in $U_q$, most of which are given in \cite{zh}. To keep the paper self-contained, each formula will be
proved unless an explicit proof can be found elsewhere.\par
For $i\in [1,m+n)\setminus m$, the  automorphism $T_{\a_i}$ of $U_q$ is defined by(see \cite[Appendix A]{zh} and also \cite[1.3]{lu1})
$$ T_{\a_i}(E_{\a_j})=\begin{cases}-F_{\a_i}K_{\a_i}, &\text{if $i=j$}\\E_{\a_j}, &\text{if $|i-j|>1$}\\-E_{\a_i}E_{\a_j}+q_i^{-1}E_{\a_j}E_{\a_i}, &\text{if $|i-j|=1$.}\end{cases}$$
$$T_{\a_i}F_{\a_j}=\begin{cases}-K_{\a_i}^{-1}E_{\a_i},&\text{if $i=j$}\\ F_{\a_j}, &\text{if $|i-j|>1$}\\ -F_{\a_j}F_{\a_i}+q_iF_{\a_i}F_{\a_j}, &\text{if $|i-j|=1$.}\end{cases}$$$$T_{\a_i}K_j=\begin{cases}K_{i+1},&\text{if $j=i$}\\K_i,&\text{if $j=i+1$}\\K_j, &\text{if $j\neq i,i+1$.}\end{cases}$$
 $T_{\a_i}$ is an  even automorphism for $U_q$, that is, $$T_{\a_i}(uv)=T_{\a_i}(u)T_{\a_i}(v),\quad  \text{for all}\quad u,v\in h(U_q).$$ \par
By a straightforward computation (\cite[A3]{zh}), one obtains for each $i\in [1, m+n)\setminus m$ the inverse map $T_{\a_i}^{-1}$:
 $$T_{\a_i}^{-1}E_{\a_j}=\begin{cases}-K_{\a_i}^{-1}F_{\a_i}, &\text{if $i=j$}\\ E_{\a_j}, &\text{if $|i-j|>1$}\\ -E_{\a_j}E_{\a_i}+q^{-1}_iE_{\a_i}E_{\a_j},& \text{if $|i-j|=1$.}\end{cases}$$
 $$T_{\a_i}^{-1}F_{\a_j}=\begin{cases}-E_{\a_i}K_{\a_i}, &\text{if $i=j$}\\F_{\a_j},&\text{if $|i-j|>1$}\\ -F_{\a_i}F_{\a_j}+q_iF_{\a_j}F_{\a_i},&\text{if $|i-j|=1$.}\end{cases}$$
$$T_{\a_i}^{-1}K_j=\begin{cases}K_{i+1},&\text{if $j=i$}\\K_i,&\text{if $j=i+1$}\\K_j, &\text{if $j\neq i,i+1$.}\end{cases}$$

It follows from the definition that
$$\begin{aligned}(b1)\quad  E_{ij}&=(-1)^{j-i-1}T_{\a_i}T_{\a_{i+1}}\cdots T_{\a_{j-1}} E_{j-1,j}\\ &=(-1)^{j-i-1}T^{-1}_{\a_{j-1}}T^{-1}_{\a_{j-2}}\cdots T^{-1}_{\a_{i+1}}E_{i,i+1},\\
(b2)\quad F_{i,j}&=(-1)^{j-i-1}T_{\a_i}T_{\a_{i+1}}\cdots T_{\a_{j-1}}F_{j-1,j}\\&=(-1)^{j-i-1}T^{-1}_{\a_{j-1}}T^{-1}_{\a_{j-2}}\cdots T^{-1}_{\a_{i+1}}F_{i,i+1}.\end{aligned}$$

By the defining relation (3), (4) and the formulas above we get $$\begin{aligned}(1)\quad E_{ij}^2&=F_{ij}^2=0,\quad (i,j)\in \mathcal I_1,\\(2)(\cite{zh})\quad \quad [E_{ij}, F_{ij}]&=\frac{K_iK^{-1}_j-K_i^{-1}K_j}{q_i-q_i^{-1}}, (i,j)\in\mathcal I.\end{aligned}$$
Let $V=V_{\0}\oplus V_{\1}$ be a vector superspace over a field $\mathbb F$. A $\mathbb F$-linear mapping $f: V\longrightarrow V$ is called $\mathbb Z_2$-graded with parity $\bar f=\bar i\in\mathbb Z_2$ if $f(V_{\bar k})\subseteq V_{\bar k+\bar i}$ for any $\bar k\in\mathbb Z_2$. Let $A=A_{\0}\oplus A_{\1}$ be an associative  $\mathbb F$-superalgebra.  A $\mathbb Z_2$-graded $\mathbb F$-linear mapping $\d$ from $A$ into itself is called a derivation if $$\d(xy)=\d (x)y+(-1)^{\bar\d\bar x}x\d(y)\quad \text{ for any}\quad x,y\in h(A).$$ Denote by $\text{Der}_{\mathbb F}A$ the set of  all derivations on $A$. For any $x,y\in h(A)$, we define $[x,y]=xy-(-1)^{\bar x\bar y}yx$.  Clearly we have $$[x,y]=-(-1)^{\bar x\bar y}[y,x].$$ For each $x\in h(A)$, it is easy to see that  $[x,-], [-,x]\in \text{Der}_{\mathbb F}A$.
\begin{lemma} (\cite{zh}) The following identities hold in $U_q$. $$\begin{aligned}(1)\quad F_{sj}F_{si}&=(-1)^{\bar F_{si}}q_sF_{si}F_{sj}, s<i<j,\\ (2)\quad F_{is}F_{js}&=(-1)^{\bar F_{js}}q_s^{-1}F_{js}F_{is}, i<j<s.\end{aligned}$$
 For $c<i<j$,
$$(3)\quad  [F_{cj}, E_{ci}]=F_{ij}K_cK_i^{-1}q_i,(4)\quad [F_{ci}, E_{cj}]=E_{ij}K_c^{-1}K_i, $$$$(5)\quad [E_{ij}, F_{cj}]=F_{ci}K_i^{-1}K_j, (6)\quad [E_{cj}, F_{ij}]=E_{ci}K_iK_j^{-1}q_i^{-1}.$$$$(7)\quad [F_{st}, F_{ij}]=-(q_j-q_j^{-1})F_{sj}F_{it}, \quad i<s<j<t.$$
\end{lemma}
\begin{proof} (1) and (2) follow from a short computation using the formulas provided by Remark (1) in Sec. 3.1. \par
(3) By Remark (1) in Sec. 3.1, we have $$[F_{cj}, E_{ci}]=[F_{ij}F_{ci}-q_iF_{ci}F_{ij}, E_{ci}].$$ Since $[-, E_{ci}]$ is a derivation on $U_q$ and $[F_{ij}, E_{ci}]=0$, we have $$[F_{cj}, E_{ci}]=F_{ij}[F_{ci}, E_{ci}]-q_i(-1)^{\bar E_{ci}\bar F_{ij}}[F_{ci},E_{ci}]F_{ij}.$$ Let us note that at least one of the $\bar E_{ci}, \bar F_{ij}$ is $\0\in\mathbb Z_2$. Then Using the formula (2) we have that $$\begin{aligned}
&[F_{cj}, E_{ci}]\\&=-(-1)^{\bar E_{ci}\bar F_{ci}}[F_{ij}\frac{K_cK_i^{-1}-K_c^{-1}K_i}{q_c-q_c^{-1}}-q_i\frac{K_cK_i^{-1}-K_c^{-1}K_i}{q_c-q_c^{-1}}F_{ij}]\\&=F_{ij}K_cK_i^{-1}q_i(-1)^{\bar E_{ci}\bar F_{ci}}\frac{q_i-q_i^{-1}}{q_c-q_c^{-1}}\\&= F_{ij}K_cK_i^{-1}q_i.\end{aligned}$$
It is easy to see that $\Omega([x,y])=[\Omega (y), \Omega (x)]$ for any $x,y\in h(U_q)$, applying which to (3) we obtain (4).\par (5),(6) can be proved similarly.\par (7) follows from an application of $\Omega$ to \cite[Lemma 4.2(6)]{z1}.
\end{proof}

\begin{lemma}\cite{z1}
$$\begin{aligned} (1)\quad [F_{ij}, F_{st}]&=0,\quad  &i<s<t<j,\\
(2)\quad [E_{ij}, F_{st}]&=0,\quad &i<s<t<j, \\ (3) \quad [F_{ij}, E_{st}]&=0, \quad &i<s<t<j.\end{aligned}$$ \end{lemma}
\begin{lemma} For $i<s<j<t$, we have $$
\begin{aligned} &(a)\quad [E_{ij}, F_{st}]&=&(q_j^{-1}-q_j)(K_sK_j^{-1})F_{jt}E_{is},\\
&(b)\quad [E_{st}, F_{ij}]&=&(q_j-q_j^{-1})F_{is}E_{jt}K_s^{-1}K_j.\end{aligned}$$
\end{lemma}
\begin{proof} It suffices to prove (a), (b) follows from the application of $\Omega$ to (a). Since $[E_{ij},-]$ is a derivation of $U_q$, we have $$\begin{aligned}
 &[E_{ij}, F_{st}]\\&=[E_{ij}, F_{jt}F_{sj}-q_jF_{sj}F_{jt}]\\&=F_{jt}[E_{ij}, F_{sj}]-q_j[E_{ij}, F_{sj}]F_{jt}\\(\text{Using Lemma 4.1(6)})&=F_{jt}E_{is}K_sK_j^{-1}q_s^{-1}-q_jE_{is}K_sK_j^{-1}q_s^{-1}F_{jt}\\&=(q_j^{-1}-q_j)(K_sK_j^{-1})F_{jt}E_{is}.
 \end{aligned}$$\end{proof}
\section{The simplicity of Kac modules}
There is an order $\prec$ defined  on the  set of elements $E_{ij}, (i,j)\in\mathcal I$(\cite{z1}):  $$E_{ij}\prec E_{st}\quad\text{if}\quad (i,j)\in\mathcal I_0\quad \text{and}\quad (s,t)\in\mathcal I_1$$ or $$ (i,j), (s,t)\in \mathcal I_{\theta}, \theta=0,1,  i<s\quad\text{or}\quad i=s\quad\text{and}\quad j<t,$$ $$F_{ij}\prec F_{st}\quad \text{if and only if}\quad E_{ij}\succ E_{st}.$$ For each $\d\in \{0,1\}^{\mathcal I_1}$, let $E_1^{\d}$ denote the product $\Pi_{(i,j)\in\mathcal I_1}E_{ij}^{\d_{ij}}$ in the order given above. Let $F^{\d}_1=\Omega (E_1^{\d})$. \par
Set $$\mathcal N_1=\langle E_1^{\d}|\d\in\{0,1\}^{\mathcal I_1}\rangle, \mathcal N_{-1}=\langle F_1^{\d}|\d\in\{0,1\}^{\mathcal I_1}\rangle,$$$$ \mathcal N_{-1}^+=\langle F_1^{\d}|\sum \d_{ij}>0\rangle, \mathcal N_1^+=\langle E_1^{\d}|\sum \d_{ij}>0\rangle.$$ By \cite{z1}, these are subalgebras of $U_q$, and $$\begin{aligned}
U_q&=\mathcal N_{-1}U_q(\g_{\0})\mathcal N_1\\&\cong \mathcal N_{-1}\otimes U_q(\g_{\0})\otimes \mathcal N_1.\end{aligned}$$ By applying the $\mathbb Z_2$-graded anti-automorphism $\Psi$, we get $$\begin{aligned}
U_q&=\mathcal N_{1}U_q(\g_{\0})\mathcal N_{-1}\\&\cong \mathcal N_1\otimes U_q(\g_{\0})\otimes \mathcal N_{-1}.\end{aligned}$$  The subalgebra $U_q(\g_{\0})\mathcal N_1$(resp. $\mathcal N_{-1}U_q(\g_{\0})$) has a nilpotent ideal $U_q(\g_{\0})\mathcal N_1^+$(resp. $\mathcal N_{-1}^+U_q(\g_{\0})$), by which each simple $U_q(\g_{\0})$-module is annihilated.  Therefore, each simple $U_q(\g_{\0})\mathcal N_1$-module can be identified with a simple $U_q(\g_{\0})$-module(cf. \cite{z1}).\par
Let $U^0$ be the subalgebra of $U_q$ generated by the elements $K_i^{\pm 1}$, $i\in [1,m+n]$. Then by the PBW theorem(\cite{z1}), $U^0$ is a polynomial algebra in variables $K_i^{\pm 1}, i\in [1,m+n]$. Let $K^{\mu}=\Pi_{i=1}^{m+n}K_i^{\mu_i}$ for $\mu=\sum_{i=1}^{m+n}\mu_i\e_i\in\Lambda$.  Denote $$X(U^0)=:\text{Hom}_{\mathbb C(q)-alg}(U^0,\mathbb C(q)).$$ Each $\l\in X(U^0)$ is completely determined by $\l(K_i)\in \mathbb C(q)^*, i\in [1,m+n]$.  Then $X(U^0)$ is an additive group with the addition defined by $$(\l_1+\l_2)(K^{\mu})=\l_1(K^{\mu})\l_2(K^{\mu}), \mu\in\Lambda.$$ Each $\l\in X(U^0)$ is called a weight for $U_q$. Note that $\Lambda$ can be canonically imbedded in $X(U^0)$ by letting $$\mu(K_i)=q_i^{\mu_i},  i\in [1,m+n], \mu=\sum_{i=1}^{m+n}\mu_i\e_i.$$ A weight in $\Lambda$ is called {\sl integral}. Clearly we have $\l(K^{\mu})=q^{(\l,\mu)}$ for $\l,\mu\in \Lambda$.\par Let $M$ be a $U_q(\g_{\0})$-module and let $\l\in X(U^0)$, let $$M_{\l}=\{x\in M|ux=\l(u)x,  u\in U^0\}.$$ A nonzero vector $v\in M_{\l}$ is called a maximal vector of weight $\l$ if $E_{ij}v^+=0$ for all $(i,j)\in \mathcal I_0$. If $M$ is finite dimensional, then $M=\sum M_{\l}$(\cite[Prop. 5.1]{j1}). If $M$ is a finite dimensional simple $U_q(\g_{\0})$-module, then there is a maximal vector, unique up to scalar multiple,  which generates $M$. In this case we denote $M$ by $M(\l)$. Regard $M(\l)$ as a $U_q(\g_{\0})\mathcal N_1$-module annihilated by $U_q(\g_{\0})\mathcal N_1^+$. Define the Kac module $$K(\l)=U_q\otimes _{U_q(\g_{\0})\mathcal N_1}M(\l).$$
Then we have $K(\l)=\mathcal N_{-1}\otimes_{\mathbb F}M(\l)$ as $\mathcal N_{-1}$-modules. \par
To study the simplicity of $K(\l)$, we define a new order on $\mathcal I_1$ by $$ (i,j)\prec (s,t)\quad\text{if $j>t$ or $j=t$ but $i<s$}.$$ We denote $(i,j)\preceq (s,t)$ if $(i,j)\prec (s,t)$ or $(i,j)=(s,t)$.  We define $F_{ij}\prec F_{st}$ if $(i,j)\prec (s,t)$.\par For each subset $I\subseteq \mathcal I_1$, denote by $F_I$ the product $\Pi_{(i,j)\in I}F_{ij}$ in the new order. In particular, we let $F_{\phi}=1$. For each $I\subseteq \mathcal I_1$, set $E_I=\Omega (F_I)$.\par
 For each $(i,j)\in\mathcal I_1$, denote  by ${> (i,j)}$(resp. $\geq (i,j)$; $<(i,j); \leq (i,j)$) the subset $$\begin{aligned} &\{(s,t)\in\mathcal I_1|(s,t)\succ (i,j)\}\\&(\text{resp.} \{(s,t)\in\mathcal I_1|(s,t)\succeq (i,j)\};\\& \{(s,t)\in\mathcal I_1|(s,t)\prec (i,j)\}; \\&\{(s,t)\in\mathcal I_1|(s,t)\preceq (i,j)\}).\end{aligned}$$ For $(i,j), (s,t)\in\mathcal I_1$ with $(i,j)\prec (s,t)$, set $$((i,j),(s,t))=\{(i',j')\in\mathcal I_1|(i,j)\prec (i',j')\prec (s,t)\}.$$
  Then we have  $$  F_{> (m,m+1)}=F_{<(1,m+n)}=1\quad\text{and}\quad F_{\mathcal I_1}=F_{<(i,j)}F_{\geq (i,j)}=F_{\leq (i,j)}F_{>(i,j)}$$ for any $(i,j)\in\mathcal I_1$.\par

 \begin{lemma}(a) $\mathcal N_{-1}$(resp. $\mathcal N_{-1}^+$) has a basis $F_I$, $I\subseteq \mathcal I_1$(resp. $\phi\neq I\subseteq \mathcal I_1$).\par (b) $\mathcal N_1$(resp. $\mathcal N_1^+$) has a basis $E_I$, $I\subseteq \mathcal I_1$(resp. $\phi\neq I\subseteq \mathcal I_1$).  \end{lemma} \begin{proof} Since $\mathcal N_1=\Omega (\mathcal N_{-1})$, (b) follows from the application of $\Omega$ to (a).\par (a). Clearly the number of the above elements is equal to $\text{dim}\mathcal N_{-1}$. We only need to show that the elements $F_I$ span $\mathcal N_{-1}$. \par
 First we claim that any product $F_{ij}F_{st}$, $(i,j), (s,t)\in\mathcal I_1$ can be written as an $\mathbb Z[q,q^{-1}]$-linear combination of products in the new order.
  The case where $j=t$ and $i>s$ follows from Lemma 4.1(2).   The cases where $j<t$ and $i\geq s$ follow from Lemma 4.1(1) and Lemma 4.2(1). The only case left is $i<s\leq m<j<t$, in which we have by Lemma 4.1(7) that $$ \begin{aligned}F_{ij}F_{st}&=-F_{st}F_{ij}-(q_j-q_j^{-1})F_{sj}F_{it}\\(\text{Using Lemma 4.2(1)})&=-F_{st}F_{ij}+(q_j-q_j^{-1})F_{it}F_{sj}.\end{aligned}$$ Thus, the claim follows.\par  Since $\mathcal I_1$ is a finite set, by  induction on the cardinality $|I|$ of $I$ we obtain that each product $\Pi_{(i,j)\in I\subseteq \mathcal I_1}F_{ij}$ in any order can be written as a $\mathbb Z[q,q^{-1}]$-linear combination of elements $F_{I'}, I'\subseteq \mathcal I_1$. \end{proof}  By the lemma, each element in $K(\l)$ is in the form $\sum_{I\subseteq \mathcal I_1} F_I\otimes v_I, v_I\in M(\l)$.
  \begin{lemma}  Let $(i,k)\in \mathcal I_1$. Then  $F_{st}F_{\geq(i,k)}=0$ for any $F_{st}\succeq F_{ik}$ or, equivalently $F_{st}F_{>(i,k)}=0$ for any $F_{st}\succ F_{ik}$.
 \end{lemma}
 \begin{proof}  Denote the set $\geq (i,k)$ by $I$. We proceed with induction on $|I|$. The case  $|I|=1$ is trivial. Assume the lemma for $|I|<d$ and consider the case $|I|=d>1$.\par  Note that $F_I=F_{i,k}F_{>(i,k)}$. By Lemma 4.1(2) and the formula (1) in the preceding section, we have $F_{st}F_I=0$ for any $(s,t)\in \mathcal I_1$ with $t=k$.\par Suppose $t<k$.  If $s\geq i$, by Lemma 4.1(2) and the formulas (1) in Sec. 4 we have $$F_{st}F_I=\pm q^{z}F_{ik}F_{st}F_{>(i,k)}=0, z\in\mathbb Z,$$ where the last equality is given by the induction hypothesis.\par If $s<i$, then we must have  $s<i\leq m<t<k$. Note that $F_{it}\succ F_{ik}$ and $F_{st}\succ F_{ik}$. Then using Lemma 4.1(7)  and the induction hypothesis we obtain $$F_{st}F_I=-F_{ik}F_{st}F_{>(i,k)}+(q_t-q_t^{-1})F_{sk}F_{it}F_{>(i,k)}=0.$$
 \end{proof}
 By a similar proof we can show that
 \begin{lemma}Let $(i,k)\in\mathcal I_1$.  Then  $F_{\leq (i,k)}F_{st}=0$ for any $F_{st}\preceq F_{ik}$ or, equivalently, $F_{< (i,k)}F_{st}=0$ for any $F_{st}\prec F_{ik}$.
 \end{lemma}
\begin{proposition} For every $(i,j)\in\mathcal I_1$, there are $z_1, z_2\in \mathbb Z$ such that $$\begin{aligned}
(1)\quad F_{ij}F_{<(i,j)}&=\pm q^{z_1}F_{\leq (i,j)}\\(2)\quad F_{>(i,j)}F_{ij}&=\pm q^{z_2}F_{\geq (i,j)}.\end{aligned}$$
\end{proposition}\begin{proof}(1) Let  $(k,s)\in\mathcal I_1$.  By Lemma 4.1, 4.2 we have,  $$F_{i,j}F_{k,s}=\begin{cases} -F_{ks}F_{i,j}, &\text{if $k<i$ and $s>j$}\\ q_kF_{ks}F_{ij},&\text {if $i=k$ and $s>j$}\\ -F_{ks}F_{i,j}+(q_j-q_j^{-1})F_{i,s}F_{kj},& \text{if $i<k<j<s$.}\end{cases}$$ Let $(k_1,s_1),\dots, (k_p,s_p)$ be all the pairs in the set $\leq (i,j)$ such that $i<k_t<j<s_t, t=1,\cdots,p$, so that $F_{i,s_t}\prec F_{k_t,s_t}$.  Then there are integers $z'_1,\dots,z'_p$ such that $$\begin{aligned}
F_{ij}F_{<(i,j)}&=F_{ij}F_{<(k_1,s_1)}F_{k_1,s_1}F_{((k_1,s_1),(i,j))}\\&=\pm q^{z'_1}F_{<(k_1,s_1)}(F_{ij}F_{k_1,s_1})F_{((k_1,s_1),(i,j))}\\&=\pm q^{z'_1}F_{<(k_1,s_1)}(-F_{k_1,s_1}F_{ij}+(q_j-q_j^{-1})F_{i,s_1}F_{k_1,j})F_{((k_1,s_1),(i,j))}\\(\text{Using Lemma 5.3} )&=\pm q^{z'_1}F_{\leq (k_1,s_1)}F_{ij}F_{((k_1,s_1),(i,j))}\\&=\cdots\\&=\pm q^{z'_p}F_{<(i,j)}F_{ij}\\&=\pm q^{z'_p}F_{\leq (i,j)}.\end{aligned}$$ (2) can be verified similarly.
\end{proof}
As an immediate consequence,  we have \begin{corollary}Let $(i,j),(s,t)\in\mathcal I_1$ with $(s,t)\preceq (i,j)$. Then
 $F_{st}F_{\leq (i,j)}=0$ .\end{corollary}
  \begin{lemma} Each nonzero submodule of  $K(\l)$ contains $F_{\mathcal I_1}\otimes v$ for some $0\neq v\in M(\l)$.
  \end{lemma}\begin{proof} Let $I,I'$ be two nonempty subsets of $\mathcal I_1$. We define $I<I'$ if, with respect to the order in $\mathcal I_1$, the first pair $(s,t)\notin I\cap I'$ is in $I'$. Then we have by Prop. 5.4 that $F_{st}F_I=\pm q^zF_{I\cap (s,t)}$ for some $z\in\mathbb Z$ and $F_{st}F_{I'}=0$.  \par
   Let $N=N_{\0}\oplus N_{\1}$ be a nonzero submodule of $K(\l)$. Take  a nonzero element $x=\sum_{I\subseteq \mathcal I_1} F_I\otimes v_I\in N$,  $v_I\neq 0$ for all $I$. Let $\bar I$ be the minimal subset appeared in the expression. \par We proceed with induction on the order of $\bar I$. If $\bar I=\mathcal I_1$, that is, $x=F_{\mathcal I_1}\otimes v$, the lemma follows. Suppose $\bar I\neq \mathcal I_1$. Let $(s,t)\in\mathcal I_1$ be the first pair such that $(s,t)\notin \bar I$. Then by definition we have  $(i,j)\in I$ for all $(i,j)\prec (s,t)$ and all $I$ appeared above. Applying $F_{st}$ to $x$ and using Prop. 5.4,  we have $F_{st}x\neq 0$, and the minimal $I$ appeared in $F_{st}x$, denoted $\bar I'$,  satisfies $\bar I'>\bar I$. Then the induction hypothesis yields the lemma.
 \end{proof}
\begin{lemma}For any $(i,j)\in \mathcal I_0$, there is $z\in \mathbb Z$ such that  $F_{ij}F_{\mathcal I_1}=q^zF_{\mathcal I_1}F_{ij}$.
 \end{lemma}
 \begin{proof}Recall that $F_{ij}=-q_cF_{ic}F_{cj}+F_{cj}F_{ic}, \quad i<c<j.$ Then it suffices to consider the case $j=i+1$.\par
 By Lemma 4.1(1), (2) and Lemma 4.2(1) we have
 $$F_{i,i+1}F_{sk}=\begin{cases} q_kF_{sk}F_{k,k+1}+F_{s,k+1}, &\text{if $i=k$}\\q_{i+1}^{-1}(F_{i+1,k}F_{i,i+1}-F_{ik}),&\text{if $s=i+1$}\\q^{ z}F_{sk}F_{i,i+1}, &\text{ otherwise,}\end{cases}$$ for some $z\in\mathbb Z$.
 Since $(i,i+1)\in\mathcal I_0$, we have that $F_{i,i+1}$ commutes, up to multiple of $q^z,z\in\mathbb Z$, with all $F_{sk}$, $(s,k)\in\mathcal I_1$, but the case $s=i+1$ if $i<m$ and the case $i=k$ if $i>m$. \par Assume $i<m$. Then we have
 $$ \begin{aligned}
 F_{i,i+1}F_{\mathcal I_1}&=F_{i,i+1}F_{<(i+1,m+n)}F_{i+1,m+n}F_{>(i+1,m+n)}\\
 &=q^{z_1}F_{<(i+1,m+n)}(F_{i,i+1}F_{i+1,m+n})F_{>(i+1,m+n)}\\
 &=q^{z_1-1}F_{<(i+1,m+n)}(F_{i+1,m+n}F_{i,i+1}-F_{i,m+n})F_{>(i+1,m+n)}\\(\text{Using $F_{<(i+1,m+n)}F_{i,m+n}=0$})
 &=q^{z_1-1}F_{\leq(i+1,m+n)}F_{i,i+1}F_{>(i+1,m+n)}\\
 &=q^{z_2}F_{<(i+1,m+n-1)}(F_{i,i+1}F_{i+1,m+n-1})F_{>(i+1,m+n-1)}\\
 &=\cdots \\&=q^zF_{\mathcal I_1}F_{i,i+1}.\end{aligned}
 $$
Similarly one verifies that $F_{i,i+1}F_{\mathcal I_1}=q^zF_{\mathcal I_1}F_{i,i+1}$ for some $z\in\mathbb Z$,  if $i>m$. This completes the proof.
 \end{proof}
 \begin{lemma}For any $(i,j)\in\mathcal I_0$, we have $E_{ij}F_{\mathcal I_1}=F_{\mathcal I_1}E_{ij}$.
 \end{lemma} \begin{proof} By the formula $E_{ij}=E_{ic}E_{cj}-q^{-1}_cE_{cj}E_{ic}$, it suffices to assume $j=i+1$. Recall the (even) derivation $[E_{i,i+1},-]$ of $U_q$. \par Using the definition of $U_q$ and Lemma 4.1(1), (2) we have, for any $(s,k)\in\mathcal I_1$, $$[E_{i,i+1}, F_{sk}]=\begin{cases}-F_{i+1,k}K_iK^{-1}_{i+1}q_{i+1},&\text{if $i=s$}\\ F_{si}K_i^{-1}K_{i+1}, & \text{if $i+1=k$}\\0,&\text{otherwise.}\end{cases}$$  Then we have $$\begin{aligned}
 &[E_{i,i+1}, F_{\mathcal I_1}]\\&=\sum _{(s,k)\in\mathcal I_1}F_{<(s,k)} [E_{i,i+1}, F_{sk}]F_{>(s,k)}\\
 &=\begin{cases}\sum_{s=i}F_{<(s,k)}(-F_{i+1,k}K_iK^{-1}_{i+1}q_{i+1})F_{>(s,k)}, &\text{if $i<m$}\\\sum_{k=i+1}F_{<(s,k)}(F_{si}K_i^{-1}K_{i+1})F_{>(s,k)},&\text{if $i>m$}\end{cases}\\&=0.\end{aligned}$$
 where the last equality is given by the fact that $F_{i+1,k}\succ F_{s,k}$ if $s=i$ and $F_{si}\succ F_{s,k}$ if $k=i+1$.
  Then the lemma follows.
 \end{proof}
 Let $E_{\mathcal I_1}=\Omega (F_{\mathcal I_1})$. Using the triangular decomposition $U_q=U_q^-\otimes U^0\otimes U^+_q$ we have $$E_{\mathcal I_1}F_{\mathcal I_1}=f(K)+\sum u^-_iu^0_iu^+_i, u_i^{\pm}\in U_q^{\pm}, f(K), u^0_i\in U^0.$$ Note that $U_q$ is a $U^0$-module under the conjugation: $$K_i\cdot u=K_iuK_{i}^{-1}, 1\leq i\leq m+n.$$ Since the $U^0$-weight of $E_{\mathcal I_1}F_{\mathcal I_1}$ is zero, we get $u_i^+=0$ if and only if $u^-_i=0$. \par Let $v_{\l}$ be a maximal vector in $M(\l)\subseteq K(\l)$. Then we get $$E_{\mathcal I_1}F_{\mathcal I_1}v_{\l}=f(K)v_{\l}=f(K)(\l)v_{\l}, f(K)(\l)\in C(q).$$ As $\l\in X(U^0)$ varies, one obtains a function $f(K)(\l)$. We denote it by $f_{m,n}(\l)$.
 \begin{proposition} The $U_q$-module $K(\l)$ is simple if and only if $f_{m,n}(\l)\neq 0$.
 \end{proposition}\begin{proof}Assume $f_{m,n}(\l)\neq 0$. Let $N=N_{\0}\oplus N_{\1}$ be a nonzero submodule of $K(\l)$. By Lemma 5.6, we have $F_{\mathcal I_1}\otimes v\in N$ for some $0\neq v\in M(\l)$. Since $K_iF_{\mathcal I_1}=q^aF_{\mathcal I_1}K_i$ for some $a_i\in \mathbb Z$, we may assume $v$ is a weight vector. Since $M(\l)$ contains a unique (up to scalar multiple) maximal vector $v_{\l}$, there is a sequence of elements $E_{\a_{i_1}}, \cdots, E_{\a_{i_s}}\in U_q(\g_{\0})$ such that $$E_{\a_{i_1}}\cdots E_{\a_{i_s}}v=v_{\l}.$$  Then Lemma 5.8 implies that  $F_{\mathcal I_1}\otimes v_{\l}\in N$,  and hence $E_{\mathcal I_1}F_{\mathcal I_1}\otimes v_{\l}=f_{m,n}(\l)\otimes v_{\l}\in N$. It follows that $v_{\l}\in N$ and hence $N=K(\l)$, so that $K(\l)$ is simple.\par Suppose $K(\l)$ is simple. By Lemma 5.7, 5.8,   the subspace $F_{\mathcal I_1}\otimes M(\l)\subseteq K(\l)$ is a $U_q(\g_{\0})$-submodule, and hence simple. Note that Coro.5.5 says that $\mathcal N_{-1}^+ F_{\mathcal I_1}\otimes M(\l)=0$, so that $F_{\mathcal I_1}\otimes M(\l)$ is a simple $\mathcal N_{-1}U_q(\g_{\0})$-module annihilated by $\mathcal N_{-1}^+U_q(\g_{\0})$. Since $K(\l)$ is simple, we have $$K(\l)=\mathcal N_1U_q(\g_{\0})\mathcal N_{-1}F_{\mathcal I_1}\otimes M(\l)=\mathcal N_1F_{\mathcal I_1}\otimes M_0(\l).$$ Since $\text{dim}\mathcal N_{-1}=\text{dim}\mathcal N_1$, we have that $K(\l)$ has a basis $$E_IF_{\mathcal I_1}\otimes v_i, I\subseteq \mathcal I_1, i=1,\dots, s,$$ with $v_1,\dots, v_s$  a basis of $M(\l)$. We can choose $v_1=v_{\l}$. Then we get $$0\neq F_{\mathcal I_1}F_{\mathcal I_1}v_{\l}=f_{m,n}(\l)v_{\l},$$ so that $f_{m,n}(\l)\neq 0$.
 \end{proof}
   \section{The polynomial $f_{m,n}(\l)$}
   This section is devoted to the determination of the polynomial $f_{m,n}(\l)$, for $\l\in X(U^0)$. Let us note that R. Zhang defined in \cite{zh} a polynomial using a different order of the product $\Pi_{(i,j)\in\mathcal I_1}F_{ij}$.   \par

 \begin{lemma} For  $1\leq i\leq m$, we have $E_{i,m+n}F_{>(i,m+n)}v_{\l}=0.$
 \end{lemma}
 \begin{proof}Using the formulas from Lemma 4.1, 4.3,  we have, for any $(s,t)\succ (i,m+n)$, $$[E_{i,m+n}, F_{st}]=\begin{cases} E_{t,m+n}K_s^{-1}K_t, &\text{if $s=i$, $t<m+n$}\\E_{is}K_sK_{m+n}^{-1}q_s^{-1},&\text{if $s>i$, $t=m+n$}\\(q_t-q_t^{-1})F_{si}E_{t,m+n}K_i^{-1}K_t,&\text{if $s<i<t<m+n$}\\0, &\text{otherwise.}\end{cases}$$
 Then we have  $$\begin{aligned}
 E_{i,m+n}F_{> (i,m+n)}v_{\l}&=[E_{(i,m+n)}, F_{> (i,m+n)}]v_{\l}\\&=\sum_{F_{st}\succ F_{i,m+n}}(-1)^{\a_{st}}F_{((i,m+n),(s,t))} [E_{i,m+n}, F_{st}]F_{> (s,t)}v_{\l}\\&=\sum_{s>i, t=m+n}(-1)^{\a_{st}} F_{((i,m+n),(s,t))} (E_{is}K_sK_{m+n}^{-1}q_s^{-1})F_{>(s,m+n)}v_{\l}\\&+\sum_{s=i,t<m+n}(-1)^{\a_{st}} F_{((i,m+n),(s,t))}( E_{t,m+n}K_s^{-1}K_t)F_{>(s,t)}v_{\l}\\&+\sum_{s<i<t<m+n}(-1)^{\a_{st}} F_{((i,m+n),(s,t))}( (q_t-q_t^{-1})F_{si} E_{t,m+n}K_i^{-1}K_t)F_{>(s,t)}v_{\l},\end{aligned}$$ where $\a_{st}\in \mathbb Z_2$. Note that the second and the third summation are equal to zero, since $E_{t,m+n}$ commutes with all $F_{ij}$($(i,j)\in\mathcal I_1$) with $F_{ij}\succ F_{st}$. \par
  We claim that the first summation is also equal to zero. In fact, we have, in the case where $s>i, t=m+n$,   $$\begin{aligned}
 E_{is}F_{>(s,m+n)}v_{\l}& =[E_{is}, F_{>(s,m+n)}]v_{\l}\\&=\sum ^{m+1}_{j=m+n-1}\sum_{k=i}^{s-1}F_{((s,m+n),(k,j))} [E_{is}, F_{kj}]F_{>(k,j)}v_{\l}.\end{aligned}$$ For $k=i, m+1\leq j\leq m+n-1$, we have by Lemma 4.1(3) that $$[E_{is}, F_{kj}]F_{>(k,j)}v_{\l}=q_sF_{sj}(K_iK_s^{-1})F_{>(k,j)}v_{\l}=0,$$ where the last equality is given by the fact that $(s,j)\succ (k,j)$.\par
  For $i<k\leq s-1$, we have by using Lemma 4.3(a) that $$[E_{is}, F_{kj}]F_{>(k,j)}v_{\l}=(q_s^{-1}-q_s)(K_kK_s^{-1})E_{ik}F_{sj}F_{>(k,j)}v_{\l}=0,$$ where the last equality follows from the fact that $(s,j)\succ (k,j)$. Thus, the claim follows.\par
 \end{proof}
  For $(i,j)\in\mathcal I$, let $K_{ij}=K_iK_j^{-1}$.   Let us denote $$[(\l+\rho)(K_{ij})]=\frac{(\l+\rho)(K_{ij})-(\l+\rho)(K_{ij}^{-1})}{q-q^{-1}}.$$ Then we see that $[(\l+\rho)(K_{ij})]=[(\l+\rho,\e_i-\e_j)]$ if $\l$ is integral.
\begin{theorem}Let $\l\in X(U^0)$. Then $f_{m,n}(\l)=\Pi_{(i,j)\in\mathcal I_1}[(\l+\rho)(K_{ij})].$ In particular, $f_{m,n}(\l)=\Pi_{(i,j)\in\mathcal I_1}[(\l+\rho,\e_i-\e_j)]$ if $\l$ is integral.
\end{theorem}\begin{proof} Using the formula (2) in Sec. 4, we have $$\begin{aligned}
E_{\mathcal I_1}F_{\mathcal I_1}v_{\l}&=E_{>(1,m+n)}(E_{1,m+n}F_{1,m+n})F_{>(1,m+n)}v_{\l}\\&=E_{> (1,m+n)}(\frac{K_{1,m+n}-K_{1,m+n}^{-1}}{q-q^{-1}})F_{>(1,m+n)}v_{\l}\\&-E_{> (1,m+n)}F_{1,m+n}E_{1,m+n}F_{> (1,m+n)}v_{\l}\\
 (\text{Using Lemma 6.1})&=E_{>(1,m+n)}\frac{K_{1,m+n}-K^{-1}_{1,m+n}}{q-q^{-1}}F_{> (1,m+n)}v_{\l}\\&=[(\l+\a_1)(K_{1,m+n})]E_{>(1,m+n)}F_{>(1,m+n)}v_{\l},\end{aligned}$$ where $\l+\a_1$ is the weight of $F_{> (1,m+n)}v_{\l}$.\par  Next we compute $E_{>(1,,m+n}F_{>(1,m+n)}v_{\l}$ in a similar way.  Continue the process,   we  get $$\begin{aligned}
 E_{\mathcal I_1}F_{\mathcal I_1}v_{\l}&=[(\l+\a_1)(K_{1,m+n})]E_{>(1,m+n)}F_{>(1,m+n)}v_{\l}\\&=[(\l+\a_1)(K_{1,m+n})][(\l+\a_2)(K_{2,m+n})]E_{>(2,m+n)}F_{>(2,m+n)}v_{\l}\\&=\cdots
 \\&=\Pi^m_{i=1}[(\l+\a_i)(K_{i,m+n})] E_{\geq(1,m+n-1)}F_{\geq (1,m+n-1)}v_{\l}, \end{aligned}$$ where $\l+\a_i$ is the weight of $F_{>(i,m+n)}v_{\l}$, $1\leq i\leq m$.   It is easily seen that $$\l+\a_i=\l-2\rho_1+\sum^i_{k=1}(\e_k-\e_{m+n}).$$  By the proof of \cite[Th.4]{z2}, we have $$(\a_i,\e_i-\e_{m+n})=(\rho, \e_i-\e_{m+n}),$$ so that $$(\l+\a_i)(K_{i,m+n})=\l(K_{i,m+n})q^{(\rho,\e_i-\e_{m+n})}=(\l+\rho)(K_{i,m+n})$$ for any $i\leq m$, which gives $$f_{m,n}(\l)=\Pi^m_{k=1}[(\l+\rho)(K_{k,m+n})] E_{\geq(1,m+n-1)}F_{\geq (1,m+n-1)}v_{\l}.$$\par
  We now prove the proposition by induction on $n$. The case $n=1$ follows immediately from the equation above. Assume the proposition for $n-1$. To proceed, let us  denote by $\rho_{m,n-1}$ the $\rho$ for Lie superalgebra $gl(m,n-1)$. By the proof of \cite[Th.4]{z2}, we have   $$(\rho_{m,n-1},\e_i-\e_j)=(\rho,\e_i-\e_j)$$ for $i<m<j\leq m+n-1$.  Applying the induction hypothesis,  we have $$\begin{aligned}
  &f_{m,n}(\l)\\&=\Pi^m_{k=1}[(\l+\rho)(K_{k, m+n})] f_{m,n-1}(\l)\\&=\Pi^m_{k=1}[(\l+\rho)(K_{k, m+n})] \Pi_{i<m<j\leq m+n-1}[(\l+\rho_{m,n-1})(K_{ij})]\\
  &=\Pi^m_{k=1}[(\l+\rho)(K_{k, m+n})] \Pi_{i<m<j\leq m+n-1}\frac{\l(K_{ij})\rho_{m,n-1}(K_{ij})-\l(K_{ij}^{-1})\rho_{m,n-1}(K_{ij}^{-1})}{q-q^{-1}}\\&=
  \Pi^m_{k=1}[(\l+\rho)(K_{k, m+n})] \Pi_{i<m<j\leq m+n-1}\frac{\l(K_{ij})q^{(\rho_{m,n-1}, \e_i-\e_j)}-\l(K_{ij}^{-1})q^{-(\rho_{m,n-1},\e_i-\e_j)}}{q-q^{-1}}\\
  &=\Pi_{(i,j)\in\mathcal I_1}[(\l+\rho)(K_{ij})].\end{aligned}$$
 \end{proof}
\section{Representations of $U_q$ at roots of unity}

\subsection{Simple $U_{\eta}$-modules}
Let $l$ be an odd number $\geq 3$ and let $\eta$ be a primitive $l$th root of unity. For $1\leq i\leq m$, let $\eta_i=\begin{cases} \eta,&\text{if $i\leq m$}\\ \eta^{-1},&\text{if $i>m$.}\end{cases}$ Set $$\mathcal A'=\{f(q)/g(q)|, f(q), g(q)\in\mathcal A, g(\eta)\neq 0\}.$$
 Let $U_{\mathcal A'}$ be the $\mathcal A'$-subalgebra of $U_q$ generated by the elements $$E_{i,i+1}, F_{i,i+1}, K_{_j}^{\pm 1},  i\in [1,m+n), j\in [1, m+n].$$

For $\psi=(\psi_{ij})\in \mathbb N^{\mathcal I_0}$, let $E^{\psi}_0$ denote the product $\Pi_{(i,j)\in\mathcal I_0}E_{ij}^{\psi_{ij}}$ in the order given in Sec.5 and let $F_0^{\psi}=\Omega (E_0^{\psi})$. Recall the notion $E_I$, $F_I$, $I\subseteq \mathcal I_1$.  Then by Lemma 5.1 and the PBW theorem of $U_q$(see \cite{z1}) we have
\begin{corollary} The $\A'$-superalgebra $U_{\A'}$ has an $\A'$-basis $$F_IF_0^{\psi}K^{\mu}E_0^{\psi'}E_{I'}, I, I'\subseteq {\mathcal I_1}, \psi,\psi'\in\mathbb N^{\mathcal I_0},\mu\in\Lambda.$$
\end{corollary}
Let $U_{\mathcal A'}(\g_{\0})$(resp. $\mathcal N_{1,\mathcal A'}$; $\mathcal N_{-1,\mathcal A'}$) be the $\mathcal A'$-subalgebra of $U_{\mathcal A'}$ generated by elements $E_{\a_i}, F_{\a_i}, K_{\a_j}^{\pm 1}, i\in [1,m+n)\setminus m, j\in[1,m+n]$(resp. $E_{ij}, (i,j)\in\mathcal I_1$; $F_{ij}, (i,j)\in\mathcal I_1$). Then we have by Sec. 5 that $$U_{\mathcal A'}=\mathcal N_{-1,\mathcal A'}U_{\mathcal A'}(\g_{\0})\mathcal N_{1,\mathcal A'}.$$ Moreover, we have from the above corollary that there is an $\mathcal A'$-module isomorphism; $$U_{\mathcal A'}\cong \mathcal N_{-1,\mathcal A'}\otimes U_{\mathcal A'}(\g_{\0})\otimes\mathcal N_{1,\mathcal A'}.$$  Lemma 5.1 says that $\mathcal N_{-1,\mathcal A'}$(resp. $\mathcal N_{1,\mathcal A'}$) has an $\mathcal A'$-basis $F_I$(resp. $E_I$), $I\subseteq \mathcal I_1$.\par Let $\mathcal N_{1,\mathcal A'}^+$(resp. $\mathcal N_{-1,\mathcal A'}^+$) be the $\mathcal A'$-submodule of $\mathcal N_1$(resp. $\mathcal N_{-1}$) generated by elements $E_I$(resp. $F_I$), $I\neq \phi$. Then by \cite{z1} $\mathcal N_{1,\mathcal A'}^+$(resp. $\mathcal N_{-1,\mathcal A'}$) is an $\mathcal A'$-subalgebra of $U_{\mathcal A'}$. Moreover, using the formulas from Sec. 4 it is easy to see that $U_{\mathcal A'}(\g_{\0})\mathcal N_{1,\mathcal A'}$ and $\mathcal N_{-1,\mathcal A'}U_{\mathcal A'}(\g_{\0})$ are $\mathcal A'$-subalgebras of $U_{\mathcal A'}$ having $U_{\mathcal A'}(\g_{\0})\mathcal N_{1,\mathcal A'}^+$ and $\mathcal N_{-1,\mathcal A'}^+U_{\mathcal A'}(\g_{\0})$ as nilpotent ideals respectively.\par
 Let $B_{\mathcal A'}$(resp. $B^-_{\mathcal A'}$; $U^0_{\mathcal A'}$) be the $\mathcal A'$-subalgebra of $U_{\mathcal A'}(\g_{\0})$ generated by elements $E_{\a_i}$, $i\neq m$(resp. $F_{\a_i}$, $i\neq m$; $K^{\pm 1}_i, 1\leq i\leq m+n$).
 By \cite[Th. 4.21]{j1}, we have $$U_{\mathcal A'}(\g_{\0})\cong B_{\mathcal A'}\otimes U^0_{\mathcal A'}\otimes B^-_{\mathcal A'}.$$ Moreover, the $\mathcal A'$-algebra $B_{\mathcal A'}$(resp. $B^-_{\mathcal A'}$) is the algebra generated by the elements $E_{\a_i}$(resp. $F_{\a_i}$), $i\neq m$ with relations (R5), (R6)(resp. (R5), (R7)).
Set $$\begin{aligned}
U_{\eta}&=U_{\mathcal A'}\otimes_{\mathcal A'} \mathbb C, &U_{\eta}(\g_{\0})&=U_{\mathcal A'}(\g_{\0})\otimes _{\mathcal A'}\mathbb C\\ \mathcal N_{-1,\eta}&=\mathcal N_{-1,\mathcal A'}\otimes \mathbb C, &\mathcal N_{1,\eta}&=\mathcal N_{1,\mathcal A'}\otimes _{\mathcal A'} \mathbb C\\\mathcal N_{1,\eta}^+&=\mathcal N_{1,\mathcal A'}^+\otimes \mathbb C,&N_{-1,\eta}^+&=\mathcal N_{-1,\mathcal A'}^+\otimes \mathbb C\\B_{\eta}&=B_{\mathcal A'}\otimes _{\mathcal A'} \mathbb C, &B^-_{\eta}&=B^-_{\mathcal A'}\otimes _{\mathcal A'} \mathbb C,\\U^0_{\eta}&=U^0_{\mathcal A'}\otimes _{\mathcal A'} \mathbb C,\end{aligned}$$ where $\mathbb C$ is viewed as an $\mathcal A'$-algebra with $q$ acting as multiplication by $\eta$. Then $U_{\eta}(\g_{\0}), \mathcal N_{\pm 1,\eta},  \mathcal N_{1,\eta}^+$ can be viewed as $\mathbb C$-subalgebras of $U_{\eta}$. We also have  $\mathbb C$-algebra isomorphisms: $$\begin{aligned} &U_{\eta}&\cong &\mathcal N_{-1,\eta}\otimes U_{\eta}(\g_{\0})\otimes \mathcal N_{1,\eta}\\ &U_{\eta}(\g_{\0})&\cong &B^-_{\eta}\otimes U^0_{\eta}\otimes B_{\eta}.\end{aligned}$$ For $x\in U_{\mathcal A}$, we denote $x\otimes 1\in U_{\eta}$ also by $x$. Then $B_{\eta}$(resp. $B^-_{\eta}$) is the algebra generated by the elements $E_{\a_i}$(resp. $F_{\a_i}$), $i\neq m$ with relations (R5), (R6)(resp. (R5), (R7)) in which $q$ is replaced by $\eta$.

   \begin{corollary}(PBW theorem)
The $\mathbb C$-superalgebra $U_{\eta}$ has a basis $$F_IF_0^{\psi}K^{\mu}E_0^{\psi'}E_{I'}, I, I'\subseteq {\mathcal I_1}, \psi,\psi'\in\mathbb N^{\mathcal I_0},\mu\in\Lambda.$$\end{corollary}
The center of the $\mathbb C$-superalgebra $U_{\eta}$ is defined by $$Z(U_{\eta})=\{x\in (U_{\eta})_{\0}|xu=ux \quad\text{for all}\quad u\in U_{\eta}\}.$$
Let $(i,j)\in\mathcal I_0$, $s\in [1,m+n]$. Then it is easy to see that $$x_{ij}=:E^l_{ij}, y_{ij}=:F_{ij}^l, z^{\pm 1}_s=:K_{s}^{\pm l}$$ are all contained in $Z(U_{\eta})$.
By the PBW theorem for $U_{\eta}$, the $\mathbb C$-subalgebra $Z_0$ generated by these elements is a polynomial algebra in variables $x_{ij}, y_{ij}, z_s^{\pm 1}$. Set $$\Lambda_l=:\{k_1\e_1+\cdots+k_{m+n}\e_{m+n}\in\Lambda |0\leq k_i<l, i=1,\cdots, m+n\}.$$Clearly we have  \begin{lemma} $U_{\eta}$ is a free $Z_0$-module having a basis $$F_IF_0^{\psi}K^{\mu}E_0^{\psi'}E_{I'}, I, I'\subseteq{\mathcal I_1}, \psi,\psi'\in\mathbb [0, l)^{\mathcal I_0},\mu\in\Lambda_l.$$\end{lemma}
Let $M=M_{\0}\oplus M_{\1}$ be a simple $U_{\eta}$-module. For any $z\in Z_0$, we define a linear mapping $$\phi_z: M\longrightarrow M,  \phi_z(x)=zx, x\in M.$$  Clearly $\phi_z$ is an even $U_{\eta}$-module homomorphism. Since  $\text{ker}\phi_z$ is a $\mathbb Z_2$-graded submodule of $M$, either  $\text{ker}\phi_z=M$ or $\text{ker}\phi_z=0$. In the former case, we have $\phi_z=0$; in the latter case, the simplicity of $M$ says that $\phi_z (M)=M$, so that $\phi_z$ is an (even) isomorphism.
\begin{lemma}(\cite[Lemma 2.1, Ch.5]{sf}) Let $R$ be a commutative ring with unity and suppose that $I\subset R$ is an ideal of $R$. Let $V$ be a finitely generated unitary $R$-module with annihilator $ann_R(V)=\{r\in R|rv=0\quad\text{for all}\quad v\in V\}$. If $IV=V$, then $I+ann_R(V)=R$.
\end{lemma}
\begin{proposition} Let $M=M_{\0}\oplus M_{\1}$ be a simple $U_{\eta}$-module. Then $M$ is finite dimensional.
\end{proposition}\begin{proof}Let $V=V_{\0}\oplus V_{\1}$ be a simple $U_{\eta}$-module. Since $U_{\eta}$ is a finitely generated $Z_0$-module by Lemma 7.3, $V$ is a finitely generated $Z_0$-module. Given any ideal $I\subseteq Z_0$, $IV$ is a $U_{\eta}$-submodule of $V$. Then either $IV=V$ or $IV=0$. Since $1\in Z_0$, $ann_{Z_0}(V)\neq Z_0$. Let $I\neq Z_0$ be any ideal containing $ann_{Z_0}(V)$. If $IV=V$, then by the above lemma we get $Z_0=ann_{Z_0}(V)+I=I$, a contradiction. Therefore, we have $IV=0$; that is $I=ann_{Z_0}(V)$, which implies that $ann_{Z_0}(V)$ is a maximal ideal of $Z_0$. By Hilbert's nullstellensatz, $Z_0/ann_{Z_0}(V)$ is finite dimensional over $\mathbb C$. Since $V$ is finite dimensional over $Z_0/ann_{Z_0}(V)$, V is finite dimensional over $\mathbb C$.
\end{proof}
\begin{lemma} For each  simple $U_{\eta}$-module $V=V_{\0}\oplus V_{\1}$, there is a $\mathbb C$-algebra homomorphism $\chi: Z_0\longrightarrow \mathbb C$ such that $(z-\chi(z))M=0$ for any $z\in Z_0$.
\end{lemma}\begin{proof}
 Let $z\in Z_0$.  Since $\mathbb C$ is algebraically closed and $V$ is finite dimensional, there is $\chi (z)\in \mathbb C$ and nonzero $v\in V$ such that $zv=\chi(z)m$. Then $$V_{\chi}=:\{v\in V|zv=\chi(z)v\}\neq 0.$$ Since $z\in (U_{\eta})_{\0}$, $V_{\chi}$ is $\mathbb Z_2$-graded. Clearly $V_{\chi}$ is a $U_{\eta}$-submodule of $V$. Thus, we have $V=V_{\chi}$; that is, $z$ acts as multiplication by $\chi(z)$ on $V$. It is routine to verify that $\chi$ defines a $\mathbb C$-algebra homomorphism $Z_0\longrightarrow \mathbb C$.\end{proof}
 Let $\chi$ be as in the lemma. Define $I_{\chi}$(resp. $I^0_{\chi}$) to be the two-sided ideal of $U_{\eta}$(resp. $U_{\eta}(\g_{\0})$) generated by the central elements $$x_{ij}-\chi(x_{ij}), y_{ij}-\chi(y_{ij}), z_s^{\pm 1}-\chi(z_s^{\pm 1}),  (i,j)\in\mathcal I_0, s\in [1,m+n].$$ Define the superalgebras $$u_{\eta,\chi}=:U_{\eta}/I_{\chi}, u_{\eta,\chi}(\g_{\0})=U_{\eta}(\g_{\0})/I^0_{\chi}.$$
  \begin{lemma}$I_{\chi}=\mathcal N_{-1,\eta}I^0_{\chi}\mathcal N_{1,\eta}.$
 \end{lemma}
 \begin{proof} Since the  elements $x-\chi(x)$, $x=x_{ij}, y_{ij}, z_s^{\pm1}$ are central in $U_{\eta}$ and all contained in $U_{\eta}(\g_{\0})$, we have $$\begin{aligned}
 I_{\chi}&=\sum _xU_{\eta}(x-\chi(x))\\& =\mathcal N_{-1,\eta}\sum_x U_{\eta}(\g_{\0})(x-\chi (x))\mathcal N_{1,\eta}\\&=\mathcal N_{-1,\eta}I^0_{\chi}\mathcal N_{1,\eta}.\end{aligned}$$

 \end{proof}
 \begin{corollary} There is a $\mathbb C$-algebra isomorphism: $u_{\eta,\chi}\cong \mathcal N_{-1,\eta}\otimes u_{\eta,\chi}(\g_{\0})\otimes \mathcal N_{1,\eta}.$
 \end{corollary}
 \begin{proof} By the lemma above, we have $$\begin{aligned}
 u_{\eta,\chi}&=U_{\eta}/I_{\chi}\\&\cong \mathcal N_{-1,\eta}\otimes U_{\eta}(\g_{\0})\otimes \mathcal N_{-1,\eta}/\mathcal N_{-1,\eta}\otimes I^0_{\chi}\otimes \mathcal N_{-1,\eta}\\&\cong \mathcal N_{-1,\eta}\otimes (U_{\eta}(\g_{\0})/I^0_{\chi})\otimes \mathcal N_{1,\eta}\\&= \mathcal N_{-1,\eta}\otimes u_{\eta,\chi}(\g_{\0})\otimes \mathcal N_{1,\eta}.\end{aligned}$$
 \end{proof}
 By Lemma 7.6, each simple $U_{\eta}$-module   is  a simple $u_{\eta,\chi}$-module for some $\chi$.  As in \cite{k2}, one can define derivations $e_{\a_i}, f_{\a_i}, k_{\pm\a_j}, i\in [1,m+n)\setminus m, j\in [1,m+n]$ of the superalgebra $U_q$ by $$e_{\a_i}=[E_{\a_i}^{(l)},-], f_{\a_i}=[F_{\a_i}^{(l)},-], k_{\pm\a_j}=[K_{\pm \a_j}^{(l)},-].$$  These derivations induces derivations on $U_{\eta}$.  By applying automorphisms of $U_{\eta}$ as that in \cite[3.5,3.6]{k1}, \cite[Th.6.1]{k2},  one can assume $\chi(x_{ij})=0$ for any $(i,j)\in\mathcal I_0$ in studying simple $U_{\eta}$-modules or simple $U_{\eta}(\g_{\0})$-modules.\par

 Assume $\chi(x_{ij})=0$ in the following. Denote by $B_{\chi}$(resp. $B^-_{\chi}; U^0_{\chi}$) the image of $B_{\eta}$(resp. $B^-_{\eta}; U^0_{\eta}$) in $u_{\eta,\chi}$.
 Since $$\begin{aligned}
 I^0_{\chi}&=\sum_x U_{\eta}(\g_{\0})(x-\chi(x))\\& =(\sum_{x=y_{ij}}B^-_{\eta}(x-\chi(x))U^0_{\eta}B_{\eta}\\&+B^-_{\eta}(\sum_{x=z_s^{\pm 1}}U^0_{\eta}(x-\chi(x))B_{\eta}\\&+B^-_{\eta}U^0_{\eta}(\sum_{x=x_{ij}}B_{\eta}(x-\chi(x)).\end{aligned}$$ By a proof similar to that in  Corollary 7.8, we get
 $$u_{\eta,\chi}\cong B_{\chi}^-\otimes U^0_{\chi}\otimes B_{\chi}. $$  In addition,  $B_{\chi}$ is the quotient of $B_{\eta}$ by the ideal generated by the central elements $E_{ij}^l, (i,j)\in \mathcal I_0$. It follows that $B_{\chi}$ is the algebra generated by the elements $E_{\a_i}, i\neq m$ and relations (R5), (R6) with $q$ replaced by $\eta$, together with $E_{ij}^l=0, (i,j)\in\mathcal I_0$.\par
 \begin{corollary}The $\mathbb C$-algebra $B_{\chi}$ is nilpotent.
 \end{corollary}\begin{proof}

 Let $G_m$ be the one dimensional multiplicative group(\cite{hu}). By the description of $B_{\chi}$ above, there is a well-defined $G_m$-action on  $B_{\chi}$ defined by $t\cdot E_{ij}=t^{j-i}E_{ij}$, $(i,j)\in\mathcal I_0$. Then $B_{\chi}$ becomes a rational $G_m$-module.  Since $B_{\chi}$ is finite dimensional,  there is a largest  $G_m$-weight $N\in\mathbb N$. It follows that any finite product $E_{i_1,j_i}\cdots E_{i_t,j_t}\in B_{\chi}$ is equal to zero, if $t>N$, since otherwise it has a $G_m$-weight $\sum_{s=1}^t(j_s-i_s)>N$. Thus, $B_{\chi}$ is nilpotent.
\end{proof}

\subsection{The simplicity of Kac modules for $u_{\eta,\chi}$}
In this section, we study $u_{\eta,\chi}$-modules. For the elements in $U_{\eta}$, we denote the images in $u_{\eta,\chi}$  by the same notation. $\chi$ is assumed to satisfy $\chi(x_{ij})=0$ for all $(i,j)\in\mathcal I_0$. Let $M=M_{\0}\oplus M_{\1}$ be a simple $u_{\eta,\chi}(\g_{\0})\mathcal N_{1,\eta}$-module. Then since $u_{\eta,\chi}(\g_{\0})\mathcal N_{1,\eta}^+$ is a nilpotent ideal of $u_{\eta,\chi}(\g_{\0})\mathcal N_{1,\eta}$, $M$ is annihilated by $u_{\eta,\chi}(\g_{\0})\mathcal N_{1,\eta}^+$. Since $$u_{\eta,\chi}(\g_{\0})\mathcal N_{1,\eta}/u_{\eta,\chi}(\g_{\0})\mathcal N_{1,\eta}^+\cong u_{\eta,\chi}(\g_{\0}),$$ $M$ is a simple $u_{\eta,\chi}(\g_{\0})$-module.
Conversely, each  $u_{\eta,\chi}(\g_{\0})$-module can be viewed as a $u_{\eta,\chi}(\g_{\0})\mathcal N_{1,\eta}$-module annihilated by $u_{\eta,\chi}(\g_{\0})\mathcal N_{1,\eta}^+$.\par
  Let $M$ be a simple $u_{\eta,\chi}(\g_{\0})$-module  annihilated by $u_{\eta,\chi}(\g_{\0})\mathcal N_{1,\eta}^+$.  Define the Kac module $$K(M)=u_{\eta,\chi}\otimes _{u_{\eta,\chi}(\g_{\0})\mathcal N_{1,\eta}} M.$$  Then we have $K(M)\cong \mathcal N_{-1,\eta}\otimes_{\mathbb C} M$ as $\mathcal N_{-1,\eta}$-modules.\par

Let $M'\subseteq M$ be a simple $U^0_{\chi}B_{\chi}$-submodule. Then $B_{\chi}M'$ is  a $U^0_{\chi}B_{\chi}$-submodule. Since $B_{\chi}$ is nilpotent, $B_{\chi}M'=0$, and hence $M'$ is a simple $U^0_{\chi}$-module. Since $U^0_{\chi}$ is commutative, we have that $M'$ is 1-dimensional. Assume $M'=\mathbb C v$. Then there is a $\mathbb C$-algebra homomorphism $\l$ from $U^0_{\chi}$ to $\mathbb C$ such that $hv=\l(h)v$ for all $h\in U^0_{\chi}$. Such an element $v\in M$ is referred to as a primitive vector of weight $\l$. We denote $X(U^0_{\chi})=\text{Hom}_{\mathbb C-alg}(U^0_{\chi},\mathbb C)$.\par
   Let $M$ be simple $u_{\eta,\chi}(\g_{\0})$-module containing a primitive vector $v_{\l}$ of weight $\l$.   Then  $M$ is spanned by elements in the form $F_IF_0^{\psi}v_{\l}$ with $\psi\in [0,l)^{\mathcal I_0}, I\subseteq \mathcal I_1$. It follows that $M=\sum_{\mu\in X(U^0_{\chi})} M_{\mu}$.  Each $x\in M_{\mu}$ is called a weight vector of weight $\mu$. \par   In the superalgebra $u_{\eta,\chi}$, from Sec. 5 we may assume $$E_{\mathcal I_1}F_{\mathcal I_1}=f(K)+\sum u^-_iu^0_iu_i^+,$$ where $u_i^{\pm}$ are in the images of $U_q^{\pm}$ in $u_{\eta,\chi}$, $f(K), u^0_i\in U^0_{\chi}$.   Then  $$E_{\mathcal I_1}F_{\mathcal I_1}v_{\l}=f(K)v_{\l}=f(K)(\l) v_{\l}.$$  Denote $f(K)(\l)$ by $f(\l)$. \par Note that all the lemmas in Sec. 5 hold in $u_{\eta,\chi}$(with $\eta$ in place of $q$) as well. By a similar argument as that in Prop. 5.9, we have\begin{proposition} $K(M)$ is a simple $u_{\eta,\chi}$-module if and only if $f(\l)\neq 0$.
\end{proposition}
 A weight $\l\in X(U^0_{\chi})$ is called {\sl integral} if $\l(K_i^{\pm 1})=\eta_i^{\pm\l_i}$ with $\l_1,\cdots,\l_{m+n}\in \mathbb Z$.  In this case, we have $\l=\l_1\e_1+\cdots +\l_{m+n}\e_{m+n}\in \Lambda.$ For each  $\a=\e_i-\e_j\in \Phi^+$, set $K_{\a}=K_iK_j^{-1}$. It is then easy to check that $\l(K_{\a})=\eta^{(\l,\a)}$ for any $\a$. Moreover, for any $K^{\mu}, \mu\in\Lambda$, we have $\l(K^{\mu})=\eta^{(\l,\mu)}$.   Then by a similar argument as that for Prop. 6.2, we have $f(\l)=\Pi_{(i,j)\in\mathcal I_1}[(\l+\rho)(K_{ij})]$, where $$[(\l+\rho)(K_{ij})]=\frac{(\l+\rho)(K_{ij})-(\l+\rho)(K_{ij}^{-1})}{\eta-\eta^{-1}}.$$ \par
  Let $M$ be a $u_{\eta,\chi}(\g_{\0})$-module. Regard   $M$  as a $u_{\eta,\chi}(\g_{\0})\mathcal N_{1,\eta}$-module annihilated by $u_{\eta,\chi}(\g_{\0})\mathcal N_{1,\eta}^+$.  Define the induced functor from the categories of $u_{\eta,\chi}(\g_{\0})$-modules to the categories of $u_{\eta,\chi}$-modules by $$\text{Ind}(M)= u_{\eta,\chi}\otimes _{u_{\eta,\chi}(\g_{\0})\mathcal N_{1,\eta}} M.$$ Clearly Ind is an exact functor and $\text{Ind}(M)=K(M)$ in case $M$ is a simple $u_{\eta,\chi}(\g_{\0})$-module.  \par
 For any $\mathcal N_{1,\eta}$-module $N=N_{\0}\oplus N_{\1}$, denote $$N^{\mathcal N_{1,\eta}^+}=\{x\in N|gx=0 \quad\text{for any}\quad g\in \mathcal N_{1,\eta}^+\}.$$ If $N$ is a $u_{\eta,\chi}$-module, it is easy to check that $N^{\mathcal N_{1,\eta}^+}$ is a ($\mathbb Z_2$-graded) $u_{\eta,\chi}(\g_{\0})\mathcal N_{1,\eta}$-submodule.
 \begin{lemma}Let $\mathcal N_{1,\eta}$ be the left-regular $\mathcal N_{1,\eta}$-module. Then $\mathcal N_{1,\eta}^{\mathcal N_{1,\eta}^+}=\mathbb C E_{\mathcal I_1}.$
 \end{lemma}
 \begin{proof} Using the anti-automorphism $\Omega$,  we need only show that $$\mathcal N_{-1,\eta}^{\mathcal N_{-1,\eta}^+}=\mathbb CF_{\mathcal I_1},$$
 for the right-regular  $\mathcal N_{-1,\eta}$-module $\mathcal N_{-1,\eta}$. Recall that $\mathcal N_{-1,\eta}$ has a basis $F_I$, $I\subseteq \mathcal I_1$. By Lemma 5.3, $F_{\mathcal I_1}F_{ij}=0$ for all $(i,j)\in \mathcal I_1$, so that $F_{\mathcal I_1}\in \mathcal N_{-1,\eta}^{\mathcal N_{-1,\eta}^+}.$ Let $x=\sum_{I\subseteq \mathcal I_1}c_I F_I\in \mathcal N_{-1,\eta}$. Suppose there is $I\subsetneqq \mathcal I_1$ with $c_I\neq 0$. Let $(i,j)$ be the largest(w.r.t the order in $\mathcal I_1$) pair not contained in some $I$ with $c_I\neq 0$. Then by Lemma 5.3 and 5.4 we have  $xF_{ij}\neq 0$. Thus $\mathcal N_{-1,\eta}^{\mathcal N_{-1,\eta}^+}=\mathbb C F_{\mathcal I_1}$.
 \end{proof}
 \begin{lemma}If $\chi (z_iz_j^{-1})^2\neq 1$ for all $(i,j)\in\mathcal I_1$, then $K(M)$ is simple for any simple $u_{\eta,\chi}(\g_{\0})$-module.
\end{lemma}
 \begin{proof} Let $v_{\l}\in M$ be a primitive vector of weight $\l$, and let $N=N_{\0}\oplus N_{\1}$ be a nonzero submodule of $K (M)$. By a similar proof as that in Lemma 5.6 we have $F_{\mathcal I_1}\otimes x\in N$ for some $0\neq x\in M$. We may assume $x$ is a weight vector of weight $\mu$. Since $M$ is a simple $u_{\eta,\chi}(\g_{\0})$-module, we have $u_{\eta,\chi}(\g_{\0})x=M$. Hence,  there is an element $$f=\sum c_iu_i^-u^0_iu_i^+\in u_{\eta,\chi}(\g_{\0})$$ such that $fx=v_{\l}$, where $u_i^-$(resp. $u^+_i$; $u^0_i$) is the product of $F_{ij}$(resp. $E_{ij}; K^{\pm 1}_s$), $(i,j)\in \mathcal I_0$, $1\leq s\leq m+n$, $c_i\in \mathbb C$.  \par Since $x$ is a weight vector, we may assume $f=\sum c_iu_i^-u_i^+$. Using Lemma 5.7 and 5.8, by a minor modification of the coefficients of $f$, we get $f'=\sum c'_iu_i^-u_i^+$, which applied to $F_{\mathcal I_1}\otimes x\in N$ to get $F_{\mathcal I_1}\otimes v_{\l}\in N$.  Applying $E_{\mathcal I_1}$ to
which we get $$\Pi_{(i,j)\in\mathcal I_1}[(\l+\rho)(K_{ij})] v_{\l}\in N.$$ Note that $K_{ij}^l=\chi(z_iz_j^{-1})$ in $u_{\eta,\chi}$, which implies that $[(\l+\rho)(K_{ij})]\neq 0$ for any $(i,j)\in\mathcal I_1$. Suppose otherwise $[(\l+\rho)(K_{ij})]=0$ for some $(i,j)\in\mathcal I_1$. Then we have $$\l(K_{ij}^2)=\rho(K_{ij}^{-2})=\eta^{-2(\rho,\e_i-\e_j)},$$ which gives $\chi(z_iz_j^{-1})^2=\l(K_{ij}^2)^l=1$, a contradiction. Then we have  $v_{\l}\in N$. Therefore $N=K(M)$, and hence $K(M)$ is simple.\end{proof}
 \begin{theorem}If $\chi (z_iz_j^{-1})^2\neq 1$ for all $(i,j)\in\mathcal I_1$, then $u_{\eta,\chi}(\g_{\0})$ and $u_{\eta,\chi}$ are Morita equivalent.
\end{theorem}
\begin{proof}
 We show that
 $K(M)^{\mathcal N_{1,\eta}^+}=M$. Note that the subspace $F_{\mathcal I_1}\otimes M\subseteq K(\l)$ is annihilated by $\mathcal N^+_{-1,\eta}$. Since $E_{ij}, F_{ij}, (i,j)\in\mathcal I_0$ commutes with $F_{\mathcal I_1}$ up to scalar multiple, the subspace is a simple $\mathcal N_{-1,\eta}u_{\eta,\chi}(\g_{\0})$-submodule of $K(M)$.
 Since $K(M)$ is simple, we have $$\begin{aligned} K(M)&=u_{\eta,\chi}F_{\mathcal I_1}\otimes M\\ & =\mathcal N_{1,\eta}F_{\mathcal I_1}\otimes M.\end{aligned}$$ Set $$K^-(F_{\mathcal I_1}\otimes M)=u_{\eta,\chi}\otimes _{\mathcal N_{-1,\eta}u_{\eta,\chi}(\g_{\0})}(F_{\mathcal I_1}\otimes M),$$ where $F_{\mathcal I_1}\otimes M$ is viewed as a $\mathcal N_{-1,\eta}u_{\eta,\chi}(\g_{\0})$-module annihilated by $\mathcal N^+_{-1,\eta}u_{\eta,\chi}(\g_{\0})$. By the comparison of dimensions we have that $K(M)$ is isomorphic to $K^-(F_{\mathcal I_1}\otimes M)$ as $u_{\eta,\chi}$-modules. Thus, as $\mathcal N_{1,\eta}$-modules, we have $$K(M)\cong \mathcal N_{1,\eta}\otimes_{\mathbb F} F_{\mathcal I_1}\otimes M,$$ from which it follows that $$\begin{aligned} K(M)^{\mathcal N_{1,\eta}^+}&\cong (\mathcal N_{1,\eta})^{\mathcal N_{1,\eta}^+}\otimes F_{\mathcal I_1}\otimes M\\&\cong E_{\mathcal I_1}F_{\mathcal I_1}\otimes M\\&=M,\end{aligned}$$ where the last equality is given by the fact that $E_{\mathcal I_1}E_{\mathcal I_1}v_{\l}\neq 0$.\par
 From above discussion, we have that the functor $(,)^{\mathcal N_{1,\eta}^+}$ is right adjoint to Ind. By a similar argument as that for \cite[Th. 3.2]{fp}, $u_{\eta,\chi}(\g_{\0})$ and $u_{\eta,\chi}$ are Morita equivalent.\end{proof}

\def\refname{
\end{document}